\documentclass{amsart}
\usepackage[utf8]{inputenc}
\usepackage{amsmath}
\usepackage{amssymb}
\usepackage{enumitem}
\usepackage{mathrsfs}
\usepackage{mathtools}
\usepackage{amsrefs}
\input xy
\xyoption{all}
\usepackage{stmaryrd} 
\usepackage[colorlinks]{hyperref}

\theoremstyle{plain} 

\newtheorem{thm}{Theorem}

\newtheorem{lemma}[thm]{Lemma}

\newtheorem{coro}[thm]{Corollary}

\newtheorem{ques}[thm]{Question}

\theoremstyle{definition}
\newtheorem{remark}[thm]{Remark}

\newcommand{\Z}{\mathbb{Z}}
\newcommand{\Q}{\mathbb{Q}}

\newcommand{\EL}{\mathrm{EL}}
\newcommand{\SL}{\mathrm{SL}}

\newcommand{\Alt}{\mathrm{Alt}}

\newcommand{\Id}{\mathsf{Id}}

\newcommand{\into}{\hookrightarrow}

\begin{document}

\author[I.~Chatterji]{Indira Chatterji}
\address{Universit\'e C\^ote d'Azur, Nice, France}
\thanks{I.C. is supported in part by ANR GALS and ANR GOFR}

\author[M.~Kassabov]{Martin Kassabov}
\address{Cornell University, USA}
\thanks{M.K. is supported in part by the Simon's Foundation grant 713557 and NSF DMS 2319371 }

\title[Property (T) and fixed point properties]{Examples of finitely presented groups with strong fixed point properties and property (T)}
\date{\today}
\begin{abstract}
    We generalize the main result in~\cite{ChK} and construct a finitely presented group with property (T) which can not act on on reasonable spaces. Such group is constructed using an generalization of Hall embedding theorem, where property (T) is added at the expense of weakening the simplicity requirement.
\end{abstract}
\maketitle

This note extends the construction in~\cite{ChK} and proves the existence of groups with strong fixed point properties, which are also finitely presented and have property (T).
\begin{thm}\label{gen}
For any non-trivial element $g\in \SL_\infty(\Z)$ there is an embedding of the group $\SL_\infty(\Z)$ into a finitely presented property (T) group $\Gamma$ that is normally generated by $g$.
\end{thm}

As explained in~\cite{ChK}, using the fact that the group $\SL_\infty(\Z)$ has only a few interesting actions, one ca deduce the following corollary:
\begin{coro}\label{noaction}
    There exists a finitely presented group with property (T) which does not act non-trivially on any ``reasonable'' space.
\end{coro}

We obtain Theorem~\ref{gen} as a special case of the more general result about recursively presented groups. 
\begin{thm}\label{gengen}
For any finitely generated recursively presented group $G$ and any non-trivial element $g\in G$, there exists an embedding of the group $G$ into a finitely presented property (T) group $\Gamma$ that is normally generated by $g$.
\end{thm}

This theorem is motivated by several results about embedding of groups into simple groups. For arbitrary groups it seems unlikely that it can be embedded into a simple finitely presented group with property (T) (because there are not that many examples such infinite groups). The condition that any given $g$ non-trivial normally generates the group $\Gamma$ can be viewed as relaxation of the condition that $\Gamma$ is simple.

Notice that for $G$ any finite group, it is straightforward to embed $G$ in a finite simple group $L$. Take for instance the action of $G$ on $G \sqcup G$ (two disjoint copies of $G$), which gives an embedding of $G$ into $\Alt(G \sqcup G)$. As a consequence, any non-trivial element $g\in G$ will normally generate $L$. According to Hall (Corollary 2 in~\cite{Hall}) a similar statement holds for any countable group $G$, as it can be embedded into a finitely generated simple group $L$. In the hyperbolic group case one can even ensure that $L$ is finitely presented~\cite{BBMZ}, and in fact the Boone-Higman conjecture predicts that any group with solvable word problem should embed in a finitely presented simple group. 
The representations of the groups resulting from these constructions are not well understood and such groups are unlikely to have property (T). 

\smallskip

The starting point of the proof of Theorem~\ref{gengen} is Higman's embedding theorem~\cite{Hig}, allowing us to reduce Theorem~\ref{gengen} to the case of finitely presented groups. (For the group  $\SL_\infty(\Z)$ one can bypass Higman`s Theorem and construct such extension explicitly, see~\cite{ChK}.)

To ensure that the group $\Gamma$ has property (T), we use a result of Ershov-Jaikin~\cite{EJ} and construct $\Gamma$ as the group $\EL_3(R)$ for some finitely generated associative ring $R$. One can use~\cite{KM} to deduce that such group%
\footnote{This requires replacing $\EL_3(R)$ with $\EL_4(R)$, which is not a real issue.}
is finitely presented, provided that the ring $R$ is finitely presented. However, our construction of the ring $R$ is quite general and we do not have any control over the presentation of the ring $R$. Instead we obtain a finitely presented group $\Gamma$, as a suitable finitely presented cover of the group  $\EL_4(R)$.

\medskip
Our first steps is an analogous to the Hall embedding theorem for rings:
\begin{lemma}\label{ringembed}
Let $R$ be a finitely generated associative ring and let $r\in R$ be a nonzero element. Assume that $R$ can be embedded into some algebra $A$ over a field $F$. Then there exists an extension $L$ together with an embedding $R \into L$ such that the ideal generated by the element $r$ is the whole ring $L$.
\end{lemma}

\begin{remark}
If the additive group of $R$ is torsion-free, then $A = R \otimes_{\Z}\Q$ will work. The torsion free assumption can not be completely removed -- 
if $p$ is a prime number such that $p^2 .1 = 0\in R$ and $r \in R$ is such that $r = p r'$. Then in any ring $L$ containing $R$, the ideal $I$ generated by $r$ is proper. Indeed $pr=p^2r'=0$ and hence any element of $I$ when multiplied by $p$ is $0$, so cannot be $1$ unless the ring $L$ is trivial.

Note that even if the ring $R$ is finitely presented we can not guarantee that the ring $L$ is finitely presented.
\end{remark}
\begin{proof}[Proof of Lemma~\ref{ringembed}]
The point of assuming that $R$ embeds in an algebra is to work on vector spaces after tensoring over $F$ and use that if $S : W \to W$ is a linear transformation of a vector space $W$ such that $\mathrm{rank}\,(S) = \dim (W)$, there exists linear transformations $P,Q : W \to W$ such that $P \circ S \circ Q = \Id$. If the dimension is finite then $S$ is invertible and one map suffices but otherwise $W$ has a free (linearly independent) subset $B$ of cardinality $\mathfrak{c}$ such that $S(B)$ is also free. Complete $B$ in a basis $B'$ and $S(B)$ in the basis $B"$, and since all those sets have the same cardinality $\mathfrak{c}$, there is a linear transformation $Q : W \to W$ extending a bijection between $B'$ and $B$ and similarly a bijection between $S(B)$ and $B'$ extends to a linear transformation $P: W \to W$ by setting $P(v) = 0$ for $v \in B'' \setminus S(B)$. One checks that the composition $P \circ S \circ Q =\Id$ as it preserves every element of the basis $B'$.

Pick an infinite cardinal $\mathfrak{c}$ larger than the dimension of the algebra $A$ over $F$. The vector space $W = A \otimes_F F^{\mathfrak{c}}$ is naturally a free and faithful left $A$-module where the action is defined by  $a\cdot( x \otimes v) = (ax) \otimes v$. This induces an embedding $\phi$ of $A$ (and hence of $R$ as well) into ${\rm End}_F(W)$. Since the element $r$ is nonzero the linear transformation $\phi(r) : W \to W$ has rank exactly  $\mathfrak{c}$ which is the same as the dimension of the vector space $W$. Hence we can find linear transformations $P,Q\in{\rm End}(W)$ such that $P\circ \phi(r)\circ Q = \Id$. Now we can take $L$ to be the subring of ${\rm End}(W)$ generated by $\phi(R)$ and $\alpha$ and $\beta$. Since $R$ is finitely generated then so is $L$ and the ideal in $L$ generated by $\phi(r)$ is the whole ring since it contains the identity.
\end{proof}
We will apply Lemma~\ref{ringembed} to rings $R$ that are a group ring $\Z[G]$ for some group $G$. In general the group $G$ does not embed into $\EL_3(\Z[G]))$, however the commutator subgroup $[G,G]$ naturally embeds into $\EL_3(\Z[G]))$. In order to avoid this problem we need to following easy result
\begin{lemma}
\label{fpcommutator}
Any finitely presented group embeds into the commutator subgroup of a finitely presented group.
\end{lemma}
\begin{proof}
Let $G=\langle \,S | R \, \rangle$ be a finitely presented group. For each generator $s \in S$ we can find a finitely presented group $G_s$ and an element $s' \in [G_s, G_s ] \subseteq G_s$ such that the cyclic groups generated by $s'$ in $G_s$ is isomorphic to the cyclic group generated by $s \in G$ (if $s$ is of infinite order we can take the $G_s$ to be the Heisenberg group).

Now, the fundamental group $\Gamma$ of the graph of groups with vertex groups $G$ and $G_s$ and edge groups isomorphic to the cyclic groups generated by each $s \in S$ is finitely presented and $G$ embeds into $\Gamma$. By construction, each generator $s$ of $G$ can be identified to an element in the commutator subgroup of $G_s$, which is inside the commutator subgroup of $\Gamma$.
\end{proof}
We now have all ingredients to finish the proof of Theorem~\ref{gengen}.
\begin{proof}[Proof of Theorem~\ref{gengen}]
According to Lemmata 3.2 and 3.4 of~\cite{ChK} and the above lemma, $G$ embeds in a finitely presented group $\bar G $ that sits as a subgroup into the commutator subgroup of a finitely presented group $\tilde G$.
Since the group $[\tilde G, \tilde G]$ embeds into $\EL_n(\Z[\tilde G])$ via $g$ in the upper left corner we have that $G \subseteq \EL_3(\Z[\tilde G])$. One can see that the normal subgroup generated by $g$ in $\EL_3(\Z[G])$ contains the elementary matrix $E_{1,2}(1-g)$ as
$$
\left(\begin{array}{ccc}
    1 & 1 & 0 \\
    0 & 1 & 0 \\
    0 & 0 & 1
\end{array}\right)
\left(\begin{array}{ccc}
    g & 0 & 0 \\
    0 & 1 & 0 \\
    0 & 0 & 1
\end{array}\right)
\left(\begin{array}{ccc}
    1 & -1 & 0 \\
    0 & 1  & 0 \\
    0 & 0  & 1
\end{array}\right)
\left(\begin{array}{ccc}
    g^{-1} & 0 & 0\\
    0      & 1 & 0 \\
    0      & 0 & 1
\end{array}\right)=
\left(\begin{array}{ccc}
    1 & 1-g & 0\\
    0 & 1   & 0\\
    0 & 0   & 1    
\end{array}\right)
$$
Now we can apply Lemma~\ref{ringembed} to $R=\Z[\tilde G]$ and $r = 1-g$ to construct a finitely generated ring $L$ containing $R$ such that the (two-sided) ideal generated by $r$ is the whole ring $L$.

By~\cite{EJ} the group $\tilde\Gamma = \EL_3(L)$ has property (T) since $L$ is finitely generated, and since $R \subset L$ we have that $\EL_3(\Z[\tilde G])$ (and therefore $G$) is embedded in $\tilde\Gamma$. By construction the normal subgroup generated by $g$ in $\tilde\Gamma$ contains $E_{1,2}(1-g)$ and hence $\tilde\Gamma$ is normally generated by $g$.

Although $\tilde\Gamma$ is finitely generated, it might not be finitely presented. To construct a finitely presented cover group $\Gamma\twoheadrightarrow\tilde\Gamma$ for Theorem~\ref{gengen}, we first use Shalom's result~\cite{Sh} to find a finitely presented cover $\Gamma_1$ with property (T). Next, since $M$ is finitely presented we can add the relations from the presentation of $M$ to obtain a quotient $\Gamma_2$ of $\Gamma_1$, which is still a cover of $\tilde\Gamma$. Now $\Gamma_2$ contains $\bar G$ and therefore $G$ as a subgroup. Finally we add finitely many relations to $\Gamma_2$ to express each generator as product of conjugates of the element $g$. That last finitely presented group $\Gamma$ has the desired properties.%
\footnote{This argument only requires the group $\tilde G$ to be finitely generated and one could get away with a weaker form of Lemma~\ref{fpcommutator}.}
\end{proof}
If we start with $G = \SL_\infty(\Z)$, the resulting group $\Gamma$ from Theorem~\ref{gengen} cannot act on any uniformly locally finite CAT(0) cell complex -- since any action of $\SL_\infty(\Z)$ on such a cell complex is trivial, which justifies Corollary~\ref{noaction}.  

The group $\SL_\infty(\Z)$ does have nontrivial actions on locally finite but not uniformly locally finite CAT(0) spaces, thus it is possible that $\Gamma$ also have nontrivial actions on such spaces, maybe even without bounded orbits.
According to Haettel in~\cite{Ha}, or Bader and Furman~\cite{BaFu}, higher rank lattices can't act non-elementarily on any hyperbolic space (that is, the action is either elliptic or parabolic). No local finiteness assumption, but it's possible that only the parabolic actions have non-local finiteness.
Combining that with Lemma 4.2 of~\cite{ChK}, we deduce that our group $\Gamma$ cannot admit a non-trivial action on a uniformly locally finite Gromov hyperbolic space. The lack of action on such a space is also a consequence of Lafforgue strong property (T)~\cite{dlS}, and the following is a natural question.
\begin{ques}
    Does $\EL_4(L)$ has Lafforgue strong property (T) (for suitable assumptions on the ring $L$) as well? What about our group $\Gamma$ of Theorem~\ref{gen}?
\end{ques}
Notice that using Kac-Moody groups Pierre-Emmanuel Caprace~\cite{Ca} provided a finitely presented group with property (T), normally generated by a copy of $\SL_3(\Z)$ and admitting a non-elementary action on a Gromov hyperbolic space. The example is as follows. Let $A$ be a generalized Cartan matrix of irreducible simply laced type, neither spherical and nor affine. Consider the ring $O = \Z[1/m]$. According to Theorem 1.1 in the paper of Ershov-Rall~\cite{ER}, if $m$ is large enough with respect to the size $d$ of the matrix $A$, then the Kac-Moody group $G_A(O)$ has property (T). The $\SL_3(\Z)$ that normally generates comes from an edge of the underlying diagram. The group $G_A(O)$ has a contracting elements according to~\cite{CF}, which according to~\cite{PSZ} produces a non-elementary action on a Gromov hyperbolic space.

\end{document}